\documentclass{conm-p-l}
\usepackage{amsmath}
\usepackage{amsthm}
\usepackage{amsfonts}
\usepackage{amssymb}
\usepackage{latexsym} 
\usepackage{xcolor}

\usepackage[matrix,tips,graph,curve]{xy}
\usepackage{hyperref}



\makeatletter 
\@addtoreset{equation}{section}
\makeatother

\theoremstyle{plain}
\newtheorem{theorem}[equation]{Theorem}
\newtheorem{corollary}[equation]{Corollary}
\newtheorem{lemma}[equation]{Lemma}
\newtheorem{proposition}[equation]{Proposition}

\theoremstyle{definition}
\newtheorem{define}[equation]{Definition}

\newtheorem{Conjecture}[equation]{Conjecture}
\newtheorem{remark}[equation]{Remark}

\newcommand{\IN}{\mathbb{N}}

\newcommand{\IR}{\mathbb{R}}

\newcommand{\tr}{\mathrm{tr}}

\newcommand{\inj}{\mathrm{inj}}

\newcommand{\Hess}{\mathrm{Hess \,}}

\newcommand\iso{{\cong}}

\def\d/{/\mspace{-6.0mu}/}

\newcommand{\p}{\partial}

\title{$L_p$-cohomology and the geometry of $p$-harmonic forms}
\author{Mark Stern}
\address{Department of Mathematics, Duke University, Durham, NC 27708}
\thanks{The author was supported in part by Simons Foundation grant 391000626.}
\subjclass[2020]{Primary 58A14; Secondary 57T15}
\begin{document}
\date{} 
\maketitle

\begin{abstract}In this note we describe basic geometric properties of $p$-harmonic forms and $p$-coclosed forms and use them to reprove vanishing theorems of Pansu and  new injectivity theorems for the $L_p$-cohomology of simply connected, pinched negatively curved manifolds. We also provide a partial resolution of a conjecture of Gromov on the vanishing of $L_p$-cohomology on symmetric spaces. 
\end{abstract}
\section{Introduction}

Let $(M^n,g)$ be a complete Riemannian $n$-manifold. 
Let 
$$L_p^k(M):= \{\text{measurable } k \text{-forms } f: f\in L_p \},$$
$$C_p^k(M):= \{  f: f\in L_p^k(M) \text{ and } df\in L_p^{k+1}(M)\},$$
$$Z_p^k(M):=\{z\in C_p^k(M): dz = 0\},$$
 and   
$$B_p^k(M) := dC_p^{k-1}(M).$$
Define the $L_p$-cohomology to be the quotient 
$$H_{p }^k(M):= Z_p^k(M)/  B_p^k(M).$$
When $M$ is compact, the de Rham cohomology  and the $L_p$-cohomology are  isomorphic, for all  $1<p<\infty$. For noncompact manifolds, however, these cohomology groups need not coincide. Moreover, in the noncompact case, for some $p$ and $k$,  $B_p^k(M)$  may not be a closed subspace of $Z_p^k(M)$. In this case, the $L_p$-cohomology  naturally decomposes into two (possibly trivial) summands: the  reduced cohomology 
$$H_{p,red}^k(M):= Z_p^k(M)/\bar B_p^k(M)  ,$$
and the (either zero or infinite dimensional non-Hausdorff) torsion cohomology
$$T_p^k(M):= \bar B_p^k(M)/B_p^k(M) ,$$ 
where $\bar B_p^k(M)$ denotes the $L_p$ closure of $B_p^k(M)$. 
 Given $1<p<\infty$, and $0\leq k\leq n$, 
we define the space of $p$-harmonic $k$-forms to be
$$\mathcal{H}_{p}^k(M):= \{h\in Z_p^k(M): d^*(|h|^{p-2}h)=0\} .$$
The $p$-harmonics play the role in $L_p$-cohomology that the $L^2$-harmonics play in $L^2$-cohomology: coupling the geometry and topology and providing simple tools for proving vanishing theorems.

The purpose of this note is to develop the elementary properties of the $p$-harmonics  and apply them to the study of $L_p$-cohomology.  
In Section 2, we collect in one place some basic elements of $L_p$-Hodge theory. Most of these somewhat elementary results are well known to analysts but perhaps less familiar to topologists. We prove the existence of canonical $L_p$-primitives, the isomorphism between $H_{p,red}^k(M)$ and $\mathcal{H}_{p}^k(M)$, and the $\mathcal{H}_{p}^k(M)\leftrightarrow\mathcal{H}_{\frac{p}{p-1}}^{n-k}(M)$ Poincar\'e duality defined by a nonlinear $L_p$-Hodge star operator. In Section 3, we derive a Bochner formula for $p$-harmonics. This formula is well known, but rarely labeled or exploited as a Bochner formula in the literature. The usual package of vanishing theorems one obtains from the Bochner formula in the $L^2$ case extends immediately to all $p\in (1,\infty)$. 

  In Section 4, we examine  the extension of the monotonicity results of \cite{DS22} to $p$-harmonics. 
Given a $p$-harmonic form, $h_p$, and $x\in M$, define for $R<\inj_M$ (the injectivity radius of $M$), 
$$I_p(h_p,x,R):= \int_{B_R(x)}|h_p|^pdv,$$  
where $B_R(x)$ denotes the geodesic ball of radius $R$ and center $x$. 
In \cite{DS22}, we showed that for $M$ compact, $\delta<R\leq \inj_M$,  and for some constant $c(\tilde M)>0$, depending only on the universal cover $\tilde M$ of $M$, we have 
\begin{align}\label{dcs2}
 \frac{b^k(M)}{vol(M)} \leq\sup_{h_2\in \mathcal{H}_{2}^k(M)\setminus\{0\}} c(\tilde M)\sup_{x\in M}\frac{\delta^{-n}I_2(h_2,x,\delta)}{I_2(h_2,x,R)} .
\end{align}
Here $b^k(M)$ denotes the kth betti number of $M$. 
If the sectional curvature of $M$, $\sec_M$, satisfies  $-1\leq\sec_M<-\frac{k^2}{ (n-k-1)^2},$ then we also showed in \cite{DS22} that 
$\frac{I_2(h_2,x,\delta)}{I_2(h_2,x,R)}$ is exponentially small in $R-\delta$. This yields 2 new proofs of the vanishing of $\mathcal{H}_{2}^k(\tilde M)$, when $M$ has residually finite fundamental group and $-1\leq\sec_M<-\frac{k^2}{ (n-k-1)^2}$. One proof follows from applying \eqref{dcs2} to a tower of covers and applying the L\"uck approximation theorem (\cite{Lu94}). The other proof (which does not require residual finiteness) follows immediately from the vanishing of $\frac{I_2(h_2,x,\delta)}{\|h_2\|_{L^2(\tilde M)}^2}= \frac{I_2(h_2,x,\delta)}{I_2(h_2,x,\infty)}=0.$

Recall that the Singer conjecture for compact negatively curved manifolds $M$ posits that 
$$\mathcal{H}_2^k(\tilde M) = 0 , \text{   for }k\not =\frac{dim(M)}{2}.$$   
Hence we are very interested in relaxing the pinching conditions required for applying \cite{DS22} to prove $\mathcal{H}_{2}^k(\tilde M)= 0$. 

 In Section 5, we show that if $-1\leq \sec_M\leq -\frac{(p-1)k}{n-k-1}$, or $-1\leq \sec_M\leq -\frac{n-k}{(k-1)(p-1)}$, then
$\frac{I_p(h_p,x,\delta)}{I_p(h_p,x,R)}$ is again exponentially small in $R-\delta$. Observe that as $p$ tends to $1$ or $\infty$, the pinching condition limits to the assumption of strictly negative curvature.  (See Theorem \ref{thmA}). Unfortunately, this does not immediately imply the Singer Conjecture, because for $p\not = 2$, the $p$-harmonic forms induce a Banach space but not a Hilbert space structure on the $L_p$-cohomology, and the Hilbert space structure is used to deduce \eqref{dcs2}. On the other hand, exponential decay of the ratio $\frac{I_p(h_p,x,\delta)}{I_p(h_p,x,R)}$ implies $\mathcal{H}_p^k(\tilde M)= 0$. This gives a 'harmonic' proof of Pansu's vanishing theorems for reduced $L_p$-cohomology of complete simply connected manifolds  for $p<\delta\frac{n-k-1}{k}+1$ and dually for $p>\frac{n-k}{\delta(k-1)}+1.$ See\cite{Pan08}.   In Section 6, we  extend these estimates  to $p$-coclosed forms (as defined in Definition \ref{pcocl}) in order to reproduce Pansu's vanishing theorems \cite{Pan08} for the torsion summand of the $L_p$-cohomology, for $p<\delta\frac{n-k}{k-1}+1$.

 In Section 7, we consider the  (necessarily reduced) $L_p$-cohomology for $p<\delta\frac{n-k}{k-1}+1$, and prove injectivity results for the natural mapping from $H_{p}^k(M)\to H_{q}^k(M)$ for $p\leq q$ and $q-p$ small.   

In Section 8,  we turn from manifolds of pinched negative curvature to symmetric spaces, and  
consider the following conjecture of Gromov \cite[Section 8C1]{Gro85}:
\begin{Conjecture}\label{ConjG0}  
If $G$ is a semisimple Lie group, with maximal compact subgroup $K$, then
$H^k_p(G/K)=0,$ for $1<p<\infty,$ and $k<\text{rank}_{\IR}(G).$
\end{Conjecture}
As an elementary application of the techniques for analyzing $p$-harmonic forms and $p$-coclosed forms developed in the preceding sections, we prove the following theorem, which  provides a partial resolution of ths conjecture.
\begin{theorem}
If $G$ is a simple Lie group, with maximal compact subgroup $K$, then 
$H^k_{p}(G/K)=0,$ for $1<p\leq 2,$ and $k<\text{rank}_{\IR}(G).$ 
\end{theorem}

\section{The nonlinear Hodge theorem and $p$-harmonic forms} 
One form of an $L_p$-Hodge theorem is a decomposition of $L_p$ forms into 
\begin{align}\overline{Im(d)}\oplus Ker(\Delta) \oplus \overline{Im(d^*)},
\end{align}
 with domains defined appropriately on $L_p$ forms.  See for example \cite{Sco95}. Such results are  fundamental, and if verified in our context would imply a significant extension of our theorems. They do not, however,  exploit many of the geometric properties that make $L_p$-cohomology interesting and distinct from $L^2$-cohomology.  

For our purposes, $p$-harmonic forms will  replace  $Ker(\Delta)$ in the above decomposition  and $p$-coclosed primitives (defined below in Definition \ref{pcocl}) will replace primitives chosen from $\overline{Im(d^*)}$. 
Much of the content of this section is already apparent in  \cite{SS70}. See also \cite[11.2]{KLV21} for a clear and careful treatment of related theorems for functions
 and \cite{ISS99} for both linear and nonlinear approaches to Hodge theory.

As a first step toward the Hodge theorem, we show the existence of canonical primitives for elements of $B_p^k(M)$. In the following, $M$ is a complete manifold, possibly with boundary. 
\begin{lemma}\label{prelim} 
For each $\zeta\in B_p^k(M)$,  
$\exists!\beta_\infty\in C_p^{k-1}(M)$ such that $d\beta_\infty = \zeta$, and  
\begin{align*}d^*\left(|\beta_\infty |^{p-2} \beta_\infty\right) = 0.
\end{align*}
\end{lemma}
\begin{proof} Write
 $d^{-1}(\zeta)$ for the set $\{\beta\in C_p^{k-1}(M):d\beta =\zeta\}.$ 
 Set $\nu(\zeta):= \inf_{\beta\in d^{-1}(\zeta)}\|\beta\|_{L_p}^p.$ Let 
$\{\beta_j\}_{j=1}^\infty\subset  d^{-1}(\zeta)$ satisfy $\lim_{j\to\infty} \|\beta_j\|_{L_p}^p=\nu(\zeta).$ Passing to a subsequence if necessary, we can assume $\{\beta_j\}_{j=1}^\infty$ converges weakly in $L_p$ to some $\beta_\infty\in L_p$. 
Since $d$ commutes with weak limits, we have $d\beta_\infty = \zeta$. Let $\phi$ be a smooth compactly supported $(k-2)$-form and $t\in \IR$. 
Then 
\begin{align}0&\leq \int_M(|\beta_\infty+td\phi|^p-|\beta_\infty|^p)dv=\int_M\int_0^1\frac{d}{ds}|\beta_\infty+std\phi|^pdsdv\nonumber\\
&= p\int_M\int_0^1 |\beta_\infty+std\phi|^{p-2}\langle \beta_\infty+std\phi,td\phi\rangle dsdv.
\end{align}
Choose $t= -\frac{1}{n}\int_M |\beta_\infty |^{p-2}\langle \beta_\infty , d\phi\rangle  dv =:\frac{c}{n},$ $n\in \IN.$
Then we have 
\begin{align}0&\leq -p\int_M |\beta_\infty |^{p-2}\langle \beta_\infty , d\phi\rangle  dv\int_M\int_0^1 |\beta_\infty+s\frac{c}{n}d\phi|^{p-2}\langle \beta_\infty+s\frac{c}{n}d\phi, d\phi\rangle dsdv.
\end{align}
Taking the limit as $n\to\infty$ yields 
\begin{align}0&\leq -p(\int_M |\beta_\infty |^{p-2}\langle \beta_\infty , d\phi\rangle  dv)^2.
\end{align}
Hence for all $\phi$ smooth and compactly supported, 
\begin{align}
0 = \int_M |\beta_\infty |^{p-2}\langle \beta_\infty , d\phi\rangle  dv,
\end{align}
and therefore
\begin{align}d^*\left(|\beta_\infty |^{p-2} \beta_\infty\right) = 0.
\end{align}
Since the $L_p$ norm is strictly convex, it has a unique critical point. Hence $\beta_\infty$ is unique. 
\end{proof}
\begin{define}\label{pcocl}We call the primitive specified in Lemma \ref{prelim} the {\em $ p$-coclosed primitive}. 
\end{define}

For $z\in Z_p^\ast(M)$, let $[z]$ and $[z]_{red}$ denote the images of $z$ in $H_{p }^\ast(M)$ and $H_{p,red}^\ast(M)$ respectively.
\begin{theorem}(Nonlinear Hodge  Theorem)
For each $\varphi\in H_{p,red}^k(M),$ $\exists ! h\in Z_p^k(M)$ such that $[h]_{red}=\varphi$, and 
\begin{align}\label{pharm}d^*(|h|^{p-2}h) = 0.
\end{align}
Moreover, $\|h\|_{L_p(M)}\leq \|z\|_{L_p(M)}$, for all $z\in  Z_p^k(M)$ such that $[z]_{red}=\varphi$, with equality only if $h=z$. 
\end{theorem}
\begin{proof}Let $z\in Z_p^k(M)$. Define 
$E_z:\bar B_p^{k}(M)\to \IR$ by 
$$E_z(\beta):= \int_M|z-\beta|^pdv.$$
Let $\rho(z):= \inf_{\beta\in \bar B_p^{k}(M)}E_z(\beta).$
Then $\rho(z) = 0$ if and only if  $[z]_{red} = 0.$ Suppose $\rho(z)\not = 0$. Let $\{\beta_j\}_{j=1}^\infty\subset \bar B_p^{k}(M)$ 
satisfy $\lim_{j\to\infty}E_z(\beta_j)= \rho(z)$. Then $\{\beta_j\}_{j=1}^\infty$ is an $L_p$ bounded sequence. Passing to a subsequence, we can assume $\{\beta_j\}_{j=1}^\infty$ is weakly convergent to some $\zeta\in \bar B_p^k(M)$.  By the weak lower semicontinuity of convex functions, 
$   E_z(\zeta)\leq \rho(z),$ and therefore $E_z(\zeta) = \rho(z).$ Set $h:= z-\zeta.$ We then have 
$$[h]_{red} = [z]_{red}.$$
The proof that $d^*(|h|^{p-2}h) = 0$ and that $h$ is the unique norm minimizer in its class is now the same as the proofs of the corresponding statements in Lemma \ref{prelim}. 
\end{proof}
\begin{corollary}
Let $M$ be a complete  riemannian manifold. Then
$$H_{p,red}^k(M)\iso \mathcal{H}_p^k(M).$$
\end{corollary}

We can also dualize our treatment of $L_p$ Hodge theory. Define 
$$\Omega_p^k(M):=\{\text{measurable } k-\text{forms } f: f\in L_p \text{ and } d^*f\in L_p\}.$$
$$F_p^k(M):=\{f\in \Omega_p^k(M):d^*f = 0\},\text{  and  }D_p^k(M):= d^*\Omega_p^{k+1}(M).$$
Define the $L_p$-homology to be 
$$Y_{ p}^k(M):= F_p^k(M)/D_p^k(M).$$
Define the torsion summand 
$$\mathcal{T}_{ p}^k(M):= \bar D_p^k(M)/D_p^k(M),$$
and the reduced homology
$$Y_{ p,\text{red}}^k(M):= F_p^k(M)/\bar D_p^k(M).$$
Define the space of $p$-coharmonic forms to be  
$$
\mathcal{E}_p^k(M):= \{f\in F_p^k(M):d(|f|^{p-2}f) = 0\}.$$
For complete manifolds, we again have $Y_{ p,\text{red}}^k(M)$ is isomorphic to $\mathcal{E}_p^k(M)$.

Analytic Poincar\'e duality proofs survive the passage from $2-$harmonic to $p$-harmonic. Assume $M$ is oriented. Let $\ast$ denote the Hodge star operator. Define 
$$\ast_ph:= \ast|h|^{p-2}h.$$
Then $|\ast_ph|^{\frac{p}{p-1}} = |h|^p,$ and 
\begin{align}\ast_{\frac{p}{p-1}} (\ast_ph) =  \ast^2 h=\pm h.
\end{align}
In particular, $\ast_ph\in \mathcal{H}_{\frac{p}{p-1}}^{n-k}(M).$ 
Interpreting Poincare duality differently, we also note that the map $h\to |h|^{p-2}h$ defines a norm preserving bijection between $\mathcal{E}_p^k(M)$ and $ \mathcal{H}_{p'}^k(M)$,
where $p' := \frac{p}{p-1}$ denotes the dual exponent. Moreover $z\to \ast z$ defines an isometric bijection between $T^k_p(M)$ and $\mathcal{T}^{n-k}_p$. 
This yields the following $L_p$-Poincar\'e duality theorem. 
\begin{theorem}\label{delduality}
Let $M^n$ be a complete  oriented riemannian manifold. Then $\ast_p$ defines a bijective norm preserving map from $ \mathcal{H}_{p}^k(M)$ to $ \mathcal{H}_{\frac{p}{p-1}}^{n-k}(M)$. 
Moreover, $h\to |h|^{p-2}h$ (respectively $z\to\ast z$) defines a norm preserving bijection between $\mathcal{E}_p^k(M)$ and $ \mathcal{H}_{p'}^k(M)$ (respectively between $T^k_p(M)$ and $\mathcal{T}^{n-k}_p$). 
\end{theorem} 
See \cite[Corollary 14]{Pan08} for a proof of Poincar\'e duality for both reduced and torsion $L_p$-cohomology, not involving $p$-harmonic forms. For the convenience of the reader, we also give an elementary proof of the duality for the torsion part of the cohomology. 
\begin{proposition}\label{tduality}
$T^k_p(M)=0$ if and only if $T^{n-k+1}_{p'}(M) = 0$. 
\end{proposition}
\begin{proof}
By the Closed Image Theorem \cite[Theorem 6.2.3]{BS18}, the image of $d(C^{k-1}_p(M))$ is closed if and only if $d^*(\Omega^k_{p'}(M))$ is closed. Using $d^*=\pm\ast d\ast$, we see that $d^*(\Omega^k_{p'}(M))$ is closed if and only if $d(\Omega^{n-k}_{p'}(M))$ is closed. 
\end{proof}
\section{The Bochner formula}\label{bochfo}

By elliptic regularity, a $p$-harmonic form $h$ is smooth in a neighborhood of any point where it is nonzero. It is not necessarily smooth, however, at any point $x_0$ where $h(x_0) = 0$, but it is H\"older continuous. (See \cite[Main Theorem]{Uhl77}). We will not repeat that difficult proof here, but we note that a first step in proving the simpler result that $|h|$ is pointwise bounded involves extending the Bochner formula to $p$-harmonic forms. The usual vanishing theorems obtained via Bochner formulas for $2-$harmonics then apply to the general $p$- harmonic case with no substantive change. 

Given a differential form $w$, let $e(w)$ denote left exterior multiplication by $w$ and let $e^*(w)$ denote the adjoint operation. Thus, for example, if $w= dF$ for some function $F$, then $e^*(w) = i_{\nabla F}$, where $i_X$ denotes the interior product with $X$.  
\begin{proposition}(Uhlenbeck \cite[Theorem 1.10]{Uhl77} )\label{bochner}
Let $h$ be a $p$-harmonic form. 
Then $h$ satisfies the Bochner formula 
\begin{align}\label{bochf}
&\Delta \frac{1}{2} |h|^{p} =  -|h|^{p-2}\left(|\nabla h|^2-  (2-p) |d|h||^2 -\langle h, e(\omega^i)e^*(\omega^j)R_{ij} h\rangle\right)\nonumber\\
&-\frac{2-p}{2p }div\left([\langle  e^*(\omega^k)   \frac{h}{|h|}, e^*(\omega^j)\frac{h}{|h|}\rangle -\langle e(\omega^j)\frac{h}{|h|}, e(\omega^k) \frac{h}{|h|}\rangle](e_k|h|^{p })e_j\right).
\end{align}
\end{proposition}
\begin{proof}
Compute in an orthonormal frame: 
\begin{align}
&-\Delta \langle h,|h|^{p-2}h\rangle = 2\langle \nabla_jh,\nabla_j|h|^{p-2}h\rangle - \langle \nabla^*\nabla h,|h|^{p-2}h\rangle - \langle h,\nabla^*\nabla |h|^{p-2}h\rangle\nonumber\\
  = &2|h|^{p-2}|\nabla h|^2- 2(2-p) |h|^{p-2}|d|h||^2 - \langle (dd^*+e(\omega^i)e^*(\omega^j)R_{ij}) h,|h|^{p-2}h\rangle\nonumber\\
& - \langle h,(d^*d +e(\omega^i)e^*(\omega^j)R_{ij}) |h|^{p-2}h\rangle\nonumber\\
  = &2|h|^{p-2}|\nabla h|^2- 2(2-p) |h|^{p-2}|d|h||^2 +\frac{2-p}{p }div(\langle  e^*(\omega^k)   \frac{h}{|h|}, e^*(\omega^j)\frac{h}{|h|}\rangle e_k(|h|^{p })e_j) \nonumber\\
&
-\frac{2-p}{p }div(\langle e(\omega^j)\frac{h}{|h|}, e(\omega^k) \frac{h}{|h|}\rangle e_k(|h|^{p })e_j)-2\langle h, e(\omega^i)e^*(\omega^j)R_{ij}  |h|^{p-2}h\rangle .
\end{align}
Hence $|h|^p$ satisfies the claimed second order equation.
\end{proof}
Integrating the $p$-harmonic Bochner formula over $M$ and applying the divergence theorem to eliminate the Laplacian term and the nonlinear divergence term  yields the same corollaries as the usual $2-$ harmonic formula. 
Define the curvature operator 
$$\mathcal{R}:=e(\omega^j)e^*(\omega^k)R_{kj}.$$
\begin{corollary}
If  $\mathcal{R} $ is positive on $k$-forms, then $H_{p,red}^k(M)=0$, {\em for all} $ p\in(1,\infty).$ If $\mathcal{R} $  is nonnegative on $k$-forms, then any $p$-harmonic $k$-form must be covariant constant, and hence zero if the volume is infinite. 
\end{corollary}
\begin{corollary}
If $M$ is complete and Ricci flat with infinite volume, $H_{p,red}^1(M)=0$, {\em for all} $ p\in(1,\infty).$ 
\end{corollary}

\begin{corollary}\label{bdd}
If $M^n$ has bounded sectional curvature and positive injectivity radius, there exists $C_M>0$ such that $\|h\|_{L^{\infty}(M)}\leq C_M\|h\|_{L^{p}(M)}$, for every $h\in \mathcal{H}^k_p(M)$,  $1<p<\infty$, $0\leq k\leq n$.  
\end{corollary}
\begin{proof}
The Bochner formula \eqref{bochf} shows that $|h|^p$ is a subsolution to an elliptic equation of divergence form with uniformly bounded coefficients. The bounded sectional curvature and positive injectivity radius then allows one to use Moser iteration  to control the sup norm of $|h|^p$ in terms of its average value on geodesic balls of radius less than or equal to $\inj_M$. (See, for example, \cite[Proposition 2.2]{LS18} for the analogous $p=2$ estimate.) This in turn implies a bound for $\|h\|_{L^\infty(M)}$ in terms of $(\inj_M)^{-\frac{n}{p}}\|h\|_{L_p}.$ 
\end{proof}

\section{Monotonicity estimates for $p$-harmonic forms} \label{monoto}
The monotonicity estimates for $p$-harmonic forms can be derived exactly as the $L^2$ case treated in \cite{DS22}. We repeat the computation here for the convenience of the reader. 

Let $h\in \mathcal{H}_{p}^k(M).$
Let $B_R$ be a geodesic ball in $M$ with boundary $S_R$. Let $r$ denote the radial function and $\p_r$ denote the radial vector field in $B_R$. 
Let $\{\lambda_a\}_a$ denote the eigenvalues of $\Hess(r)$. Choose an orthonormal eigenframe $\{e_a\}_a$ with dual coframe $\{\omega^a\}_a$, so that 
$ e(\omega^i)e^*(\omega^j) \Hess(r)_{ji} =  \sum_a\lambda_ae(\omega^a)e^*(\omega^a).$  

Define the functions $w_p(h,r)$ and $\mu_p(h,r)$ by 
\begin{align}\label{defw} w_p(h,r):= \frac{\int_{S_r} |h|^{p}  \sum_a \lambda_a(\frac{1}{p}-|e^*(\omega^a) \frac{h}{|h|} |^2)d\sigma}{\int_{S_r} |h|^{p} d\sigma},
\end{align}
and 
\begin{align}\label{defmu} \mu_p(h,r):= \frac{\int_{S_r} |h|^{p}  |e^*(dr) \frac{h}{|h|} |^2d\sigma}{\int_{S_r} |h|^{p} d\sigma}.
\end{align}
\begin{proposition}\label{pricey}
With $w_p,\mu_p$, and $h$ as defined above, 
\begin{align}\label{int0}
& \int_{B_\sigma}w_p(h,r)|h|^pdv = e^{-\int_\sigma^\tau\frac{pw_p(h,s)}{1-p\mu_p(h,s)}\frac{ds}{s}}\int_{B_\tau}w_p(h,r)|h|^pdv.
\end{align}
\end{proposition}
\begin{proof} 

 We use the fact that on a Riemannian manifold, we can express the Lie derivative acting on forms in terms of the covariant derivative:
\begin{align}\label{Lie}
L_{\p_r}= di_{\p_r}+i_{\p_r}d = \nabla_{\p_r} + e(\omega^i)e^*(\omega^j)\Hess(r)_{ji}.
\end{align}
By Stoke's theorem we have 
\begin{align}\label{stokes}
\int_{B_R}\langle L_{\p_r}h,|h|^{p-2}h\rangle dv = \int_{S_R}|h|^{p-2}|e^*(dr)h|^2d\sigma, 
\end{align}
and by \eqref{Lie}, 
\begin{align}\label{nabla}
&\int_{B_R}\langle L_{\p_r}h,|h|^{p-2}h\rangle dv =  \int_{B_R}\langle (\nabla_{\p_r} + e(\omega^i)e^*(\omega^j)\Hess(r)_{ji})h,|h|^{p-2}h\rangle dv \nonumber\\
&=  \int_{B_R}(  \frac{1}{p}\frac{\p|h|^p}{\p r} + \langle e(\omega^i)e^*(\omega^j)\Hess(r)_{ji})h,|h|^{p-2}h\rangle) dv\nonumber\\
&=  \int_{B_R} (  \frac{1}{p}|h|^p \Delta r + \langle e(\omega^i)e^*(\omega^j)\Hess(r)_{ji})h,|h|^{p-2}h\rangle )dv
+ \frac{1}{p} \int_{S_R}|h|^{p} d\sigma.
\end{align}
Combining \eqref{stokes} and \eqref{nabla} yields 
\begin{align}\label{nabla2}
& \int_{S_R}(\frac{1}{p}-\mu_p(h,R))|h|^{p}  d\sigma =  \int_{B_R} w_p(h,r)|h|^p    dv.
\end{align}
Setting $Y(t):= \int_{B_t}w_p(h,r)|h|^pdv,$ we rewrite \eqref{nabla} as 
\begin{align}\label{ode}
&  (\frac{1}{p}-\mu_p(h,R))Y' =  w_p(h,R)Y(R).
\end{align}
Integrate this expression to obtain 
\begin{align}\label{int}
& \int_{B_\sigma}w_p(h,r)|h|^pdv = e^{-\int_\sigma^\tau\frac{pw_p(h,s)}{1-p\mu_p(h,s)}\frac{ds}{s}}\int_{B_\tau}w_p(h,r)|h|^pdv,
\end{align}
as claimed. 
\end{proof}
\begin{proposition}\label{incond}
Let $h$ be a $p$-harmonic $k$-form. Then with $w_p$ and $\mu_p$ defined by \eqref{defw} and \eqref{defmu}, we have 
\begin{align}\label{wlim}\lim_{r\to 0}rw_p(h,r) =   \frac{n-1-pk+p\lim_{r\to 0}\mu_p(h,r)}{p }.
\end{align}
If $h(o)\not = 0$, then 
\begin{align}\label{mulim}\lim_{r\to 0}\mu_p(h,r) =   \frac{k}{n}, 
\end{align}
and
\begin{align}\label{muwlim}\lim_{r\to 0}rw_p(h,r) =   (n-1)(\frac{1}{p }-\frac{k}{n}).
\end{align}
\end{proposition}
\begin{proof}For $h$ smooth in a neighborhood of $o$, equation \eqref{mulim} is \cite[Lemma 18]{DS22}. 
By elliptic regularity, however, $h$ is smooth in a neighborhood of any point at which it is nonzero, and \eqref{mulim} follows. 

 $\Hess(r)$ converges to  the  Euclidean Hessian of $r$, as $r\to 0$. Hence all $\lambda_a$ approach $\frac{1}{r}$ and \eqref{wlim} follows. Finally, \eqref{muwlim} is an immediate consequence of \eqref{mulim} and \eqref{wlim}.  
\end{proof}
\begin{corollary}\label{zeroday} If 
$p<\frac{n}{k}$ and $w_p(h,s)\geq 0$, for $s\in(0,R]$, then $\frac{1}{p}-\mu_p(h,s)> 0$, for $s\in(0,R]$. 
If 
$p>\frac{n}{k}$ and $w_p(h,s)\leq 0$, for $s\in(0,R]$, then $\frac{1}{p}-\mu_p(h,s)<0$, for $s\in(0,R]$.
\end{corollary}
\begin{proof}By \eqref{nabla2}  $\frac{1}{p}-\mu_p(h,s)$ does not change sign, given the hypotheses on $w_p$,  and Proposition \ref{incond} fixes the sign.
\end{proof}
In order to exploit Proposition \ref{pricey}, we need geometric input. In the next section, we translate pinched negative curvature into lower bounds for $w_p$. 

\section{ Negative Sectional Curvature}
Let $M^n$ be a nonnegatively curved n-manifold with $-1\leq\sec_M\leq -\delta^2,$ $\delta\geq 0$, where $\sec_M$ denotes the sectional curvature of $M$. Let $r$ again denote the radial function of some geodesic ball, $B_R$, in $M$. By the Rauch comparison theorem \cite[Theorem 6.4.3]{Pet16}, the eigenvalues $\{\lambda_a\}_{a } $ of $\Hess(r)$ satisfy
\begin{align}\label{II}
\delta\coth(\delta r)\leq \lambda_a \leq\coth(r),
\end{align}
(except for the zero eigenvalue associated to the radial direction)
and $\tr \Hess(r)$   satisfies 
\begin{align}\label{mean}
\tr \Hess(r)\geq (n-1)\delta\coth(r).
\end{align}
In a neighborhood of some point $x_0\in B_R\setminus\{0\}$, let $\{\omega^a\}_{a=1}^{n-1}\cup\{dr\}$ be a local orthonormal coframe   such that the $\omega^a$ are metrically dual to Hessian eigendirections with eigenvalue $\lambda_a$. Let $h$ be a $k$-form, with $|h|(x_0)\not = 0$. 
Let $J\subset\{1,\ldots,n-1\}$ denote the indices of Hessian eigendirections such that at $x_0$, $|e^*(\omega^a) \frac{h}{|h|} |^2>\frac{1}{p}.$ Since $h$ is a $k$-form, at $x_0$ we have
\begin{align}
|e^*(dr) \frac{h}{|h|} |^2+\sum_{a=1}^{n-1}|e^*(\omega^a) \frac{h}{|h|} |^2=k.
\end{align}
Then
\begin{align} &\sum_a \lambda_a(\frac{1}{p}-|e^*(\omega^a) \frac{h}{|h|} |^2)\nonumber\\
&= \delta\coth(\delta r)(\frac{n-1}{p}-k+|e^*(dr) \frac{h}{|h|} |^2)
+\sum_a (\lambda_a-\delta\coth(\delta r))(\frac{1}{p}-|e^*(\omega^a) \frac{h}{|h|} |^2)\nonumber\\
&\geq  \delta\coth(\delta r)(\frac{n-1}{p}-k+|e^*(dr) \frac{h}{|h|} |^2)
+\sum_{a\in J} (\coth(r)-\delta\coth(\delta r))(\frac{1}{p}-|e^*(\omega^a) \frac{h}{|h|} |^2)\nonumber\\
&\geq  \delta\coth(\delta r)(\frac{n-1}{p}-k )
+ (\coth(r)-\delta\coth(\delta r))( \frac{|J|}{p} -\min\{|J|,k\})  \nonumber\\
&\geq  \coth(  r)[\delta(\frac{n-1}{p}-k )
+ (1-\delta )( \frac{k}{p} -k)]  \nonumber\\
&=  \frac{k}{p} \coth(  r)[\delta \frac{n-k-1}{k}+ 1 -p  ] .
\end{align}
 This is positive for 
$p< \delta \frac{n-k-1}{k}+ 1$   (which is less than $\frac{n}{k}$). In the quarter pinched case, $\delta =\frac{1}{2}$, and for $k\leq \frac{n-1}{2}$ $p< \delta \frac{n-k-1}{k}+ 1$ specializes to $p<  \frac{3}{2}.$
Observe that the conjugate exponent to $\delta \frac{n-k-1}{k}+ 1$ is $1+\frac{ k}{\delta (n-k-1)} ,$ (where by conjugate exponent, we mean the $p'(=\frac{p}{p-1})$ satisfying $\frac{1}{p}+\frac{1}{p'}= 1$). 

For large $p$, we   let 
$L$ denote the eigendirections such that $|e^*(\omega^a) \frac{h}{|h|} |^2<\frac{1}{p}.$ Then
\begin{align} &\sum_a \lambda_a(|e^*(\omega^a) \frac{h}{|h|} |^2-\frac{1}{p})\nonumber\\
&= \delta\coth(\delta r)(k-\frac{n-1}{p}- |e^*(dr) \frac{h}{|h|} |^2)
+\sum_a (\lambda_a-\delta\coth(\delta r))(|e^*(\omega^a) \frac{h}{|h|} |^2-\frac{1}{p})\nonumber\\
&\geq  \delta\coth(  r)(k-1-\frac{n-1}{p} )
+\sum_{a\in L} (\lambda_a-\delta\coth( r))(|e^*(\omega^a) \frac{h}{|h|} |^2-\frac{1}{p})\nonumber\\
&\geq \coth(  r) (n-k)\frac{p-1}{p}[\frac{\delta(k - 1)}{n-k}    -   \frac{1}{p-1} ].
 \end{align}
This is positive for ($\frac{n}{k}< $)  
$ \frac{n-k}{\delta (k - 1)}+1 <  p.$ For $\delta = \frac{1}{2}$ and $k\geq \frac{n+1}{2}$ this gives 
$  3 < p.$ Note that $3$ is the conjugate exponent to $\frac{3}{2}$. More generally the limiting exponents for $k$ and $n-k$ are conjugate for fixed $\delta$. 

We have included the parenthetical comparisons to $\frac{n}{k}$ for easy application of Corollary \ref{zeroday}.

Combining these positivity results with the estimates of the preceding section yields the following theorem. 
\begin{theorem}\label{thmA} Let $M$ be a complete Riemannian manifold with $-1\leq \sec_M\leq -\delta^2< 0.$ 
Let $h$ be a $p$-harmonic $k$-form on $M$. Let $ \sigma\leq \tau\leq \inj_M$. If $p<\delta \frac{n-k-1}{k}+ 1$, 
then $w_p(h,r)$ is bounded below by a positive constant independent of $h$ and $r$, and 
\begin{align}\int_{B_\sigma}w_p(h,r)|h|^pdv\leq e^{-(\tau-\sigma) k  [\delta \frac{n-k-1}{k}+ 1 -p  ] }\int_{B_\tau}w_p(h,r)|h|^pdv.
\end{align}
If $p>\frac{n-k}{\delta (k - 1)}+1$, 
then 
\begin{align}\int_{B_\sigma}w_p(h,r)|h|^pdv\leq e^{-(\tau-\sigma) k  (p-1)[\frac{\delta(k - 1)}{p}    -   \frac{(n-k)}{p(p-1)} ]}\int_{B_\tau}w_p(h,r)|h|^pdv.
\end{align}
\end{theorem}
These monotonicity estimates give us a new proof of Pansu's reduced $L_p$-cohomology vanishing theorem for simply connected complete manifolds of pinched negative curvature. The fact that we obtain precisely the same pinching constraints is not surprising as both proofs at some point rely on estimates for geometric quantities related to the eigenvalues of $\Hess(r)$. 
\begin{corollary}\label{coroA}(Pansu)
Let $M^n$ be a simply connected complete Riemannian n-manifold with $-1\leq \sec_M\leq -\delta^2< 0.$ 
  If $k<\frac{n}{p}$, and 
 $p<\delta \frac{n-k-1}{k}+ 1$, or $k>\frac{n}{p}$, and $p>\frac{n-k}{\delta (k - 1)}+1$,
then 
\begin{align}\mathcal{H}^k_p(M) =0 = H^k_{p,red}(M).
\end{align}
 
\end{corollary} 
\begin{proof}Let 
$h\in \mathcal{H}^k_p(M)$. Take the limit as
$\tau\to\infty$ in Theorem \ref{thmA} to deduce \\ $\int_{B_\sigma}w_p(h,r)|h|^pdv = 0$, for all $\sigma$ and all balls. Since $w_p(h,r)\not = 0$ for nonzero $h$ and $r$ sufficiently small, we obtain a contradiction unless $h=0$. 
\end{proof}

\begin{remark}
In the $p=2$ case, exponential bounds of the form given in Theorem \ref{thmA}  imply for $M$ compact that $\frac{b^k(M)}{vol(M)}$ is $O(e^{-\inj_M k  [\delta \frac{n-k-1}{k}+ 1 -p  ] })$, and therefore imply the vanishing of the $k^{th}$ $L^2$ Betti number of $M$. The desire to exploit these bounds for such Betti number estimates was  our initial motivation for extending these monotonicity estimates to $1<p<\infty$. Unfortunately, more work remains to translate these estimates into Betti number bounds. 
\end{remark}

In the next section we extend the monotonicity estimates developed for $p$-harmonic forms to obtain vanishing theorems for the non-Hausdorff summand of the cohomology, even though that summand does not admit $p$-harmonic representatives. Once again we obtain the same pinching constraints as Pansu. 

\section{Closure of exact forms}
In this section we consider the torsion $T^k_p(M)$. We give a new proof of  a vanishing theorem of Pansu (see \cite{Pan08}). 
\begin{theorem}\label{thmT}(Pansu)
Let $M$ be a simply connected complete Riemannian manifold with $-1\leq \sec_M\leq -\delta^2< 0.$ 
  If $k-1<\frac{n}{p}$, and 
 $p<\delta \frac{n-k}{k-1}+ 1$,
then 
\begin{align}T^k_p(M) =0.
\end{align}
 
\end{theorem} 
\begin{proof}
 Let $\{z_j\}_{j=1}^\infty\subset B_p^k(M)$ be an $L_p$ Cauchy sequence, with $z_j\stackrel{L_p}{\to }z\in Z_p^k(M).$ Let $ b_j$ denote   the  $p$-coclosed primitive of   $z_j$. (See Lemma \ref{prelim}.)  
Repeating the first steps of the derivation of our monotonicity estimates given in Section \ref{monoto}, we have, since $db_j = z_j$, 
\begin{align*}\int_{B_R}\langle \{d,e^*(dr)\}b_j,|b_j|^{p-2}b_j\rangle dv = \int_{S_R}|e^*(dr)b_j|^2|b_j|^{p-2} d\sigma+\int_{B_R}\langle  e^*(dr)z_j,|b_j|^{p-2}b_j\rangle dv,
\end{align*}
and
\begin{align*}\int_{B_R}\langle \{d,e^*(dr)\}b_j,|b_j|^{p-2}b_j\rangle dv = -\int_{B_R}w_p(b_j,r)|b_j|^{p}  dv +\int_{S_R} \frac{1}{p}|b_j|^pd\sigma.
\end{align*}
Combined, these yield 
\begin{align} \label{coclprice} &\int_{B_R} w_p(b_j,r)|b_j|^{p}  dv   =  \int_{S_R} (\frac{1}{p}-\mu_p(b_j,R))|b_j|^pd\sigma -\int_{B_R}\langle  e^*(dr)z_j,|b_j|^{p-2}b_j\rangle dv.
\end{align}
Our pinching hypotheses imply  $w_p(b_j,r)\geq \kappa_{k-1}>0$, for some positive constant $\kappa_{k-1}$.  Hence, taking the limit as $R\to\infty$, the $S_R$ integral vanishes, and H\"older's inequality yields
\begin{align}  \kappa_{k-1}\|b_j\|_{L_p}  \leq  \|z_j\|_{L_p}. 
\end{align}
Since $\{z_j\}_{j=1}^\infty$ is Cauchy, $\{b_j\}_{j=1}^\infty$ is bounded. Passing to a subsequence, we can assume $b_j$ coverges weakly in $L_p$ to some $b.$
Since $d$ commutes with weak limits, $db = z$. Hence $\bar B_p^k/B_p^k = \{0\}.$
\end{proof}

\section{Injectivity and pinched negative sectional curvature}
 Developing the techniques of the previous section, we next prove injectivity results for natural maps from $H^k_{p,red}(M)\to  H^k_{q}(M),$ for $p<\delta\frac{n-k}{k-1}+1$, $p<q$, and $q-p$ small.

By Corollary \ref{bdd}, $p$-harmonic forms with finite $L_p$ norm are pointwise bounded on complete manifolds $N$ with bounded sectional curvature and injectivity radius bounded below. Hence for $p\leq q$,  $\mathcal{H}^k_{p}(N)\subset Z_q^k(N),$ and we can define a map 
$$i_{p,q}^k:H^k_{p,red}(N)\stackrel{\text{Hodge isomorphism}}{\to}\mathcal{H}^k_{p}(N)\stackrel{\text{inclusion}}{\to}  Z^k_{q}(N)\stackrel{\text{quotient}}{\to}  H^k_{q}(N).$$ 
\begin{proposition}
Let $M$ be a complete, simply connected, negatively curved n-manifold, with $-1\leq \sec_M\leq-\delta^2<0$.
If $p<1+\delta\frac{n-k}{k-1}$ and $\frac{q-p}{q}<\delta(n-k)-(p-1)(k-1)$,\\
then $i_{p,q}^k$ is injective.
\end{proposition}
\begin{proof}
Let $h$ be a $p$-harmonic representative of an element of the kernel of $i_{p,q}^k$. Fix a point $o\in M$, and let $B_R$ denote the geodesic ball of radius $R$ and center $o$. Let $b_R$ denote the $p$-coclosed primitive of $h$ restricted to $\bar B_R$. By \eqref{coclprice} we have for $T\leq R$, 
\begin{align} \label{coclpriceR} &\int_{B_T} w_p(b_R,r)|b_R|^{p}  dv   =  \int_{S_T} (\frac{1}{p}-\mu_p(b_R,T))|b_R|^pd\sigma -\int_{B_T}\langle  e^*(dr)h,|b_R|^{p-2}b_R\rangle dv.
\end{align}
Suppose for  $\epsilon:= \frac{1}{2p}(\delta(n-k)-(p-1)(k-1)-\frac{q-p}{q})$, there exists a sequence $\{(T_j,R_j)\}_j$ 
with $T_j\to\infty$ and $R_j\geq T_j$, so that 
$\int_{B_{T_j}} |b_{R_j}|^{p}  dv\leq \frac{1}{\epsilon^p}\int_{B_{T_j}} |h|^{p}  dv.$ Then passing to a subsequence, we can find $b_\infty\in L_p$ such that the $b_{R_j}$ converge weakly to $b_\infty$ on compact sets. Therefore $db_\infty = h$, which implies $h=0$. Hence, for $h\not = 0$, $\exists T_0>0$ so that 
\begin{align}\label{est1}\int_{B_{T}} |h|^{p}  dv\leq \epsilon \int_{B_{T}} |b_{R_j}|^{p}  dv, \forall T>T_0.
\end{align}
Setting $Y_R(T):= \int_{B_T} (w_p(b_R,r)-\epsilon)|b_R|^{p}  dv$, we combine \eqref{coclpriceR} and \eqref{est1} to obtain 
for $T>T_0$,
\begin{align*}
 &Y_R(T)   \leq    \frac{(\frac{1}{p}-\mu_p(b_R,T))}{w_p(b_R,T)-\epsilon}\frac{dY_R(T)}{dT}.
\end{align*}
Integrate this differential inequality from $T_0$ to   $R$ to get 
\begin{align} \label{coclpriceint} 
&Y_R(R) \geq e^{\int_{T_0}^R\frac{w_p(b_R,T)-\epsilon}{(\frac{1}{p}-\mu_p(b_R,T))} dT }Y_R(T_0)\geq e^{(R-T_0)[\delta(n-k)-(p-1)(k-1)-p\epsilon] }Y_R(T_0).
\end{align}
On the other hand, since $h\in Ker(i^k_{p,q}),$ $\exists \beta_q\in L^q$ such that $d\beta_q = h$. Since $b_R$ is the $L_p(B_R)$ norm minimizing primitive for $h_{|B_R}$,
we have 
\begin{align}\label{contraest}
\|b_R\|_{L_p(B_R)}^p&\leq \|\beta_q\|_{L_p(B_R)}^p  \leq \|\beta_q\|_{L^q(M)}^{\frac{p}{q}}vol(B_R)^{\frac{q-p}{q}}\nonumber\\
&\leq \|\beta_q\|_{L^q(M)}^{\frac{p}{q}}vol(B_R)^{\frac{q-p}{q}} \leq c(n)\|\beta_q\|_{L^q(M)}^{\frac{p}{q}}e^{\frac{q-p}{q}R}.
\end{align} 
As $R\to \infty$, \eqref{contraest} and \eqref{coclpriceint} are mutually contradictory, unless $h=0$. Hence the kernel of $i_{p,q}^k$ is $\{0\}$, as claimed.  
\end{proof}

\section{Symmetric spaces}\label{symm}
The curvature pinching hypotheses in the preceding sections are somewhat coarse,  as they depend on the Rauch comparison theorem. In more homogeneous geometries, we can often compute Lie derivatives of useful vector fields explicitly and can therefore achieve sharper results. As illustration, we resolve part of a conjecture of Gromov. 

In this section, $G$ is a semisimple Lie group of noncompact type, $K$ is a maximal compact subgroup, and $X:= G/K$ is the associated symmetric space. Let $A$ be a maximal real split torus of $G$. (So $\text{rank}_{\IR}(G) = dim(A).$)  Let $\mathfrak{g}$,  $\mathfrak{k}$, and $\mathfrak{a}$ denote the Lie algebras of $G$, $K$ and $A$ respectively. Let $\mathfrak{g} = \mathfrak{p}\oplus\mathfrak{k}$ denote the Cartan decomposition of $\mathfrak{g}$ determined by $K$.

\begin{Conjecture}\label{conjG}(Gromov) 
$H^k_p(X)=0,$ for $1<p<\infty,$ and $k<\text{rank}_{\IR}(G).$
\end{Conjecture}
 
 A choice of Weyl chamber determines a set $\Phi_G^+$ of positive roots of $G$ relative to $A$. When $G$ is not real split, extend $\mathfrak{a}$ to a Cartan subalgebra  $\mathfrak{h}$ of $\mathfrak{g}$, and let $\hat \Phi_G^+$ denote the set of roots of $\mathfrak{h}$  whose restriction to $\mathfrak{a}$ lie in $\Phi_G^+$. Let $N$ be the unipotent subgroup of $G$ whose Lie algebra $\mathfrak{n}$ is the direct sum of the root spaces with root in $\hat \Phi^+_G$. Let $\{\tilde N_\beta\}_{\beta\in \hat\Phi_G^+}$ be an orthonormal basis of $\mathfrak{n}$ with respect to the metric determined by the Killing form and the Cartan involution. Let $\{\tilde \omega_\beta\}_{\beta\in \hat\Phi_G^+}$ denote the dual basis. We identify elements of the Lie algebras of $A$ and $N$  with  left invariant vector fields on $A$, and $N$ respectively, in the usual fashion.   With this notation, we have the Iwasawa decomposition  $G=NAK$, and $X$ is isometric to $N\times A$ equipped with the metric 
$$g_X := da^2+\sum_{\beta\in \hat \Phi^+_G}a^{-2\beta}(\tilde\omega^\beta )^2,$$
where $da^2$ is a Euclidean metric on $A$. Orthonormalize our $TN$ subframe with respect to this metric, setting  $ N_\beta:= a^\beta \tilde N_\beta,$ and $\omega^\beta:= a^{-\beta}\tilde\omega^\beta.$

We now adapt our monotonicity machinery to the symmetric setting. 
Let $T$ be a unit vector field in the Lie algebra $\mathfrak{a}$ of $A$.    Then 
\begin{align}\nabla_{N_\beta}T = -\beta(T)N_\beta.
\end{align}
Hence we can write the Lie derivative with respect to $T$ acting on forms as 
\begin{align}
L_T = di_T+i_Td = \nabla_T -\sum_{\beta\in \hat\Phi^+_G} \beta(T)e(\omega^\beta)e^*(\omega^\beta).
\end{align}

\begin{align}\label{jet2}
0=\int_{X}\langle L_Th,|h|^{p-2}h\rangle dv =  \int_{X} \sum_{\beta\in \Phi^+_G} \beta(T)|h|^{p-2}(\frac{1}{p}|h|^2-|e^*(\omega^\beta)h|^2 )dv 
.
\end{align} 

\begin{proposition}\label{precursor}
If there exists a  unit vector $T\in\mathfrak{a}$   and some $\delta>0$, so that for all $k$-forms $h$,  
\begin{align}\label{condition}\sum_{\beta\in \hat\Phi^+_G} \beta(T) (\frac{1}{p}|h|^2-|e^*(\omega^\beta)h|^2 )\geq \delta |h|^2,
\end{align}
then $\mathcal{H}_p^k(X) = 0.$
\end{proposition}
\begin{proof}
Suppose $T$ satisfying the hypotheses exists. Let $h$ be a $p$-harmonic $k$-form. By \eqref{jet2}, we have  
\begin{align}
\delta\int_{X} |h|^{p} dv\leq  \int_{X} \sum_{\beta\in \hat\Phi^+_G} \beta(T)|h|^{p-2}(\frac{1}{p}|h|^2-|e^*(\omega^\beta)h|^2 )dv =0.
\end{align}
 Hence $h=0$. 
\end{proof}

The preceding proposition can be used to prove many vanishing theorems. We explore its application to Conjecture \ref{conjG}.  
\begin{theorem}\label{ThmG} 
Assume $G$ is simple. Then 
$\mathcal{H}_p^k(X) = 0$, for $k<\text{rank}_{\IR}(G)$ and $1<p\leq 2$. 
\end{theorem}
\begin{proof}
The vanishing of $\mathcal{H}_2^k(X)$, for $k\not = \frac{\text{dim}(X)}{2}$ is well known. See \cite{Bor85}, and more recently, see  for example \cite{Olb02}. So, we are left to treat the case $p<2$. 
First assume  $G$ is real split.  Let $\mu$ denote the maximal root in $\Phi^+_G$. 
Let $T= T_G$ be the vector in $\mathfrak{a}$ which is metrically dual to $\frac{\mu}{|\mu|}$. Then for $\beta\in \Phi^+_G$, 
$\beta(T)\in\{0,\frac{|\mu|}{2},|\mu|\}$, with $|\mu|$ only achieved for $\beta = \mu$. 
Let $n_G(1):=|\{\beta\in\Phi^+_G:\beta(T) = 1\}|,$ and  $n_G(2):=|\{\beta\in\Phi^+_G:\beta(T) = 2\}|.$ Then $n_G(2) = 1$, for $G$ real split. Hence  
\begin{align}\label{roots1}\sum_{\beta\in\Phi^+_G} \frac{\beta(T)}{p}|h|^2 = \frac{(n_G(1)+2)|\mu|}{2p} |h|^2.\end{align}
 On the other hand,  
\begin{align}\label{roots2}\sum_{\beta\in \Phi^+_G}  \beta(T) |e^*(\omega^\beta)h|^2 \leq  (k+1)\frac{|\mu|}{2}|h|^2.
\end{align}
Hence condition \eqref{condition} is satisfied if $\frac{(n_G(1)+2) }{ p} >  (k+1) .$ In other words, if 
\begin{align}\label{concrete}
p<\frac{n_G(1)+2}{k+1}.
\end{align}

Case 1) For the $A_n$ root system, in the usual notation,  $\Phi^+_G$ is given by $\lambda_i-\lambda_j$, $1\leq i <j\leq n+1$,  and $\mu = \lambda_{1}-\lambda_{n+1}$. Hence $n_G(1) = 2n-2$ and condition \eqref{condition} holds for 
\begin{align}\label{A_n}
p<\frac{2n}{k+1} . 
\end{align}  
Case 2) For the $B_n$ root system, $n\geq 2$,  $\Phi^+_G$ is given by $\lambda_i\pm \lambda_j$, $1\leq i <j\leq n$, and $\lambda_j$, $1\leq j\leq n$, and  $\mu = \lambda_1+\lambda_2$. Then $n_G(1) = 4n-6$, and condition \eqref{condition} holds for 
\begin{align}
p<\frac{4n-4}{k+1} . 
\end{align}  
Case 3) For the $C_n$ root system, $n\geq 2$,  $\Phi^+_G$ is given by $\lambda_i\pm \lambda_j$, $1\leq i <j\leq n$, and $2\lambda_j$, $1\leq j\leq n$, and   $\mu = 2 \lambda_1 $. Then $n_G(1)=2n-2$, and condition \eqref{condition} holds for 
\begin{align}
p<\frac{2n}{k+1} . 
\end{align}  
Case 4) For the $D_n$ root system, $n\geq 4$,  the roots of $N$ are given by $\lambda_i\pm \lambda_j$, $1\leq i <j\leq n$,  and $\mu = \lambda_1+\lambda_2$. Then $n_G(1) = 4n-8$, and condition \eqref{condition} holds for 
\begin{align}
p<\frac{4n-6}{k+1} . 
\end{align}  
Case 5) For $G=G_2,$ $n_G(1) = 4$, and condition \eqref{condition} holds for 
\begin{align}
p<\frac{6}{k+1} . 
\end{align}
Case 6) For $G=F_4,$ $n_G(1) = 14$, and condition \eqref{condition} holds for 
\begin{align}
p<\frac{16}{k+1} . 
\end{align}
Case 7) For $G=E_6,$ $n_G(1) = 20$, and condition \eqref{condition} holds for 
\begin{align}
p<\frac{22}{k+1} . 
\end{align}
Case 8) For $G=E_7,$ $n_G(1) = 32$, and condition \eqref{condition} holds for 
\begin{align}
p<\frac{34}{k+1} . 
\end{align}
Case 9) For $G=E_8,$ $n_G(1) = 56$, and condition \eqref{condition} holds for 
\begin{align}
p<\frac{58}{k+1} . 
\end{align}
For $k+1\leq \text{rank}_{\IR}(G)$, condition \eqref{condition} thus holds in all cases for $p<2$. 

When $G$ is simple but not real split, the restricted root system is one of the three classical systems of type $A,B,$ or $C$, or the exceptional system $F_4$, but the multiplicity of the roots can increase, and for some short roots $\beta$, $2\beta$ may occur as a root. See \cite{Ara62}, especially pp. 32-33, for this and subsequent properties of restricted root systems.  
Then \eqref{roots1} and \eqref{roots2} become 
\begin{align}\label{roots11}\sum_{\beta\in\hat\Phi^+_G} \frac{\beta(T)}{p}|h|^2 = \frac{(n_G(1)+2n_G(2))|\mu|}{2p} |h|^2,\end{align}
 and
\begin{align}\label{roots22}\sum_{\beta\in \hat\Phi^+_G}  \beta(T) |e^*(\omega^\beta)h|^2 \leq  (k+\min\{k,n_G(2)\})\frac{|\mu|}{2}|h|^2 \leq  2k\frac{|\mu|}{2}|h|^2.
\end{align}
When $n_G(2)>1$, we often employ the simplified vanishing criterion
\begin{align}\label{condo1}
p<\frac{n_G(1)+2n_G(2) }{2k} 
\end{align}
instead of the sharper vanishing condition:
\begin{align}\label{condo2}
p<\frac{n_G(1)+2n_G(2) }{ k+\min\{k,n_G(2)\}} .
\end{align}
For simpler comparison with our real split computations, let $n_{A_n}(j), n_{B_n}(j),n_{C_n}(j) $, $j=1,2$ denote $n_G(j)$ for $G$ a real split simple Lie group of type $A_n,B_n$, or $C_n$, respectively. If $G$ has restricted root system of type $A_n$, $n>1$,  and is not real split, then $n_G(2) \geq 2$, and $n_G(1)\geq 2n_{A_n}(1)= 4n-4.$ 
Hence condition \ref{condo1} holds for 
$$p<  \frac{2n}{k}.$$
If $G$ has restricted root system of type $B_n$, $n>1$,  and is not real split, then \ref{condo1} holds for $p<2$, since it holds in the real split case, and $n_{B_n}(1)+2n_{B_n}(2) \leq n_G(1)+2n_G(2).$ Similarly, \ref{condo1} holds for $p<2$, for all restricted root sytems of type $F_4$. 
So, we are left to consider the $C_n$ case. Araki lists 4 different restricted root systems of type $C_n$, with 
\begin{enumerate}
\item $(n_G(1),n_G(2)) = (2n_{C_n}(1),1) =(4n-4,1)$, and $\frac{n_G(1)+2n_G(2) }{k+1} = \frac{4n-2  }{k+1}$.
\item $(n_G(1),n_G(2)) = (4n_{C_n}(1),3) =(8n-8,3)$, and $\frac{n_G(1)+2n_G(2) }{2k} = \frac{8n-2}{2k}$.
\item $(n_G(1),n_G(2)) = (4n_{C_n}(1),1) =(8n-8,1)$, and $\frac{n_G(1)+2n_G(2) }{k+1} = \frac{8n-6}{k+1}$.
\item $(n_G(1),n_G(2)) = (8n_{C_n}(1),1) =(16n-16,1)$, and $\frac{n_G(1)+2n_G(2) }{k+1} = \frac{16n-14}{k+1}$. 
\end{enumerate}
In all these $C_n$ instances, the vanishing criterion holds for $p<2$, if $k<n$. 
\end{proof}
The vanishing of the torsion part of the cohomology can be proved in a similar fashion. 
\begin{theorem}\label{clean}
Let $G$ be a simple Lie group.  Then 
$T_p^k(X) = 0$, for $p< \frac{n_G(2)+2n_G(1)}{k-1 +\min\{k-1,n_G(2)\}}$. In particular, $T_p^k(X) = 0$, for $k\leq \text{rank}_{\IR}(G)$ and $p\leq 2$ 
\end{theorem}
\begin{proof}
Let $T_G$ be the vector field defined in the proof of Theorem \ref{ThmG}. Let  $\{z_j\}_{j=1}^\infty\subset B_p^k(X)$ be an $L_p$ Cauchy sequence, 
with $z_j\stackrel{L_p}{\to }z\in Z_p^k(X).$ Let $ b_j$ denote   the  $p$-coclosed primitive of   $z_j$. Then arguing as in Section 6, with $T_G$ replacing the radial vector field, we have 
\begin{align}\label{coclprice2} 
&\int_{X} \sum_{\beta\in \Phi^+_G} \beta(T_G)|b_j|^{p-2}(\frac{1}{p}|b_j|^2-|e^*(\omega^\beta)b_j|^2 )dv  =   -\int_{X}\langle  i_{T_G}z_j,|b_j|^{p-2}b_j\rangle dv.
\end{align}
Hence $\|b_j\|_{L_p}\leq \frac{|T_G|}{\delta}\|z_j\|_{L_p},$ when 
\begin{align}\label{condo3}\sum_{\beta\in \Phi^+_G} \beta(T_G)|b_j|^{p-2}(\frac{1}{p}|b_j|^2-|e^*(\omega^\beta)b_j|^2)\geq \delta|b_j|^2, \text{ for some  } \delta>0.
\end{align}  Arguing again as in Section 6, we deduce that $\{b_j\}_j$ is uniformly bounded in $L^p$ and therefore admits a subsequence that converges to a primitive for $z$. Hence $z$ is exact, and therefore 
 $ B_p^k(X)$ is closed. By \eqref{condo2}, condition \eqref{condo3} holds when  
$p<\frac{n_G(2)+2n_G(1)}{k-1 +\min\{k-1,n_G(2)\}}$. As we have recorded in the proof of Theorem \ref{ThmG}, when $k\leq rank_{\IR}G$, $p<\frac{n_G(2)+2n_G(1)}{k-1 +\min\{k-1,n_G(2)\}}$ holds for $p\leq 2$, for all simple Lie groups, except when $p=2$ for the real split Lie groups of type $A_n$ and $C_n$.  These can be handled by a Bochner type argument, which we include below. We treat the $A_n$ case. The $C_n$ case can be treated in the same manner. 

Consider an $(n-1)$-form $\phi$ on $X = SL_{n+1}(\IR)/SO(n+1).$  Let $t^j$ denote Euclidean coordinates on $A$, and let $\mu$ still denote the maximal root in $\Phi^+_{SL_{n+1}}$. Write 
\begin{align}\label{bndrycase}&\sum_{\beta\in \Phi^+_G} \beta(T_G) (\frac{1}{2}|\phi|^2-|e^*(\omega^\beta)\phi|^2) \nonumber\\
&= \frac{n_G(1)+2}{2}|\phi|^2-|e^*(\omega^\mu)\phi|^2-(n-1)|\phi|^2+
\sum_{\alpha:\alpha(T_G)=0}|e^*(\omega^\alpha)\phi|^2+ \sum_j|e^*(dt^j)\phi|^2\nonumber\\
&=  |\omega^\mu\wedge\phi|^2 + \sum_{\alpha:\alpha(T_G)=0}|e^*(\omega^\alpha)\phi|^2+ \sum_k|e^*(dt^k)\phi|^2,
\end{align}
which vanishes  only if $\phi = \omega^\mu\wedge \phi_1,$ where $\phi_1$ is annihilated by all $e^*(dt^j)$ and all $e^*(\omega^\alpha)$, with $\alpha(T_G) = 0$. Combining \eqref{bndrycase} with the $p=2$ case of \eqref{coclprice2} yields 
\begin{align}\label{coclprice3} 
&\int_{X} \sum_k|e^*(dt^k)b_j|^2dv  =   -2\int_{X}\langle  i_{T_G}z_j,b_j\rangle dv.
\end{align}
Hence 
\begin{align}\label{coclprice4} 
& \sum_k\|e^*(dt^k)b_j\|^2_{L^2}  \leq 2\|z_j\|_{L^2}\|b_j\|_{L^2}.
\end{align}
We will combine this with an estimate which is good when $\|e^*(dt^j)\phi\|^2_{L^2}<\epsilon^2\| \phi\|^2_{L^2}, $ some small $\epsilon$, in order to obtain estimates for the primitives, $b_j$. Expand 
\begin{align}
\nabla_{N_\beta} =:\nabla_{N_\beta}^0-\beta(\frac{\p}{\p t^j})(e(\omega^\beta)e^*(dt^j)-e(dt^j)e^*(\omega^\beta)),
\end{align}
where $\nabla_{N_\beta}^0$ maps the smooth span of $\{dt^j\}_j$ to itself and the smooth span of $\{\omega^\beta\}_\beta$ to itself. 
Write 
\begin{align}
d = e(dt^j)(\nabla_{\frac{\p}{\p t^j}}-\sum_{\beta\in \hat\Phi_G^+}\beta(\frac{\p}{\p t^j})e(\omega^\beta)e^*(\omega^\beta))+ e(\omega^\beta)\nabla_{N_\beta}^0 =: d_A+d_N,
\end{align}
and 
\begin{align}d^*= -e^*(dt^j)(\nabla_{\frac{\p}{\p t^j}}-\sum_{\beta\in \hat\Phi_G^+}\beta(\frac{\p}{\p t^j})e^*(\omega^\beta)e(\omega^\beta))- e(\omega^\beta)^*\nabla_{N_\beta}^0 =: d_A^*+d_N^*.
\end{align}
Then
\begin{align}\{d_A^*,d_N\} &= \{-e^*(dt^j)(\nabla_{\frac{\p}{\p t^j}}-\sum_{\beta\in \hat\Phi_G^+}\beta(\frac{\p}{\p t^j})e^*(\omega^\beta)e(\omega^\beta)),e(\omega^\alpha)\nabla_{N_\alpha}^0\}\nonumber\\
 &= -\sum_{\beta\in \hat\Phi_G^+}e^*(dt^j) \beta(\frac{\p}{\p t^j}) e(\omega^\beta)\nabla_{N_\beta}^0.
\end{align}
Expand
\begin{align*}
&\|d\phi\|^2_{L^2}+\|d^*\phi\|^2_{L^2}= \|d_A\phi+d_N\phi\|^2_{L^2}+\|d_A^*\phi+d_N^*\phi\|^2_{L^2}\nonumber\\
=&
 \|d_A\phi\|^2_{L^2}+\|d_N\phi\|^2_{L^2}+2\langle \phi,\{d_A^*,d_N\}\phi\rangle_{L^2}+\| d_A^*\phi\|^2_{L^2}+\| d_N^*\phi\|^2_{L^2}\nonumber\\
= &
 \|e(dt^j)(\nabla_{\frac{\p}{\p t^j}}-\sum_{\beta\in \hat\Phi_G^+}\beta(\frac{\p}{\p t^j})e(\omega^\beta)e^*(\omega^\beta))\phi\|^2_{L^2}+\| e^*(dt^j)(\nabla_{\frac{\p}{\p t^j}}-\sum_{\beta\in \hat\Phi_G^+}\beta(\frac{\p}{\p t^j})e^*(\omega^\beta)e(\omega^\beta))\phi\|^2_{L^2}\nonumber\\
+&\|d_N\phi\|^2_{L^2}-2\sum_{\beta\in \hat\Phi_G^+}\beta(\frac{\p}{\p t^j})\langle  \nabla_{N_\beta}^0 e^*(\omega^\beta)\phi,  e^*(dt^j)\phi\rangle_{L^2}+\| d_N^*\phi\|^2_{L^2}\nonumber\\
=&
 \sum_j\| \nabla_{\frac{\p}{\p t^j}} \phi\|^2_{L^2}+\|\sum_{\beta\in \hat\Phi_G^+}e(dt^j) \beta(\frac{\p}{\p t^j})e(\omega^\beta)e^*(\omega^\beta) \phi\|^2_{L^2}-2\langle  \nabla_{\frac{\p}{\p t^j}}\phi,\sum_{\beta\in \hat\Phi_G^+}\beta(\frac{\p}{\p t^j})e(\omega^\beta)e^*(\omega^\beta) e^*(dt^j)\phi\rangle_{L^2}\nonumber\\
+&
 \| \sum_{\beta\in \hat\Phi_G^+}e^*(dt^j) \beta(\frac{\p}{\p t^j})e^*(\omega^\beta)e(\omega^\beta) \phi\|^2_{L^2}
-2\langle e^*(dt^j) \nabla_{\frac{\p}{\p t^j}}\phi,\sum_{\beta\in \hat\Phi_G^+}\beta(\frac{\p}{\p t^j})e^*(\omega^\beta)e(\omega^\beta))\phi\rangle_{L^2}\nonumber\\
+&\|d_N\phi\|^2_{L^2}-2\sum_{\beta\in \hat\Phi_G^+}\beta(\frac{\p}{\p t^j})\langle \nabla_{N_\beta}^0 e^*(\omega^\beta)\phi,  e^*(dt^j)\phi\rangle_{L^2}+\| d_N^*\phi\|^2_{L^2}\nonumber\\
\geq &
  \sum_j\| \nabla_{\frac{\p}{\p t^j}} \phi\|^2_{L^2}+  \|\sum_{\beta\in \hat\Phi_G^+} \beta(T_G)e(\omega^\beta)e^*(\omega^\beta) \phi\|^2_{L^2}\nonumber\\
+&
\|d_N\phi\|^2_{L^2}+\| d_N^*\phi\|^2_{L^2}-c_n (\|\nabla\phi\|_{L^2}\sqrt{\sum_j\|e^*(dt^j)\phi\|^2_{L^2}}+\sum_j\|e^*(dt^j)\phi\|^2_{L^2})\nonumber\\
\geq &
  \sum_j\| \nabla_{\frac{\p}{\p t^j}} \phi\|^2_{L^2}+  n^2\| \phi\|^2_{L^2}+\|d_N\phi\|^2_{L^2}+\| d_N^*\phi\|^2_{L^2}\nonumber\\
-&
c_n (\|\nabla\phi\|_{L^2}\sqrt{\sum_j\|e^*(dt^j)\phi\|^2_{L^2}}+\sum_j\|e^*(dt^j)\phi\|^2_{L^2}),
\end{align*}
for some positive constant $c_n$. Specializing to $\phi = b_m$, with $d^*b_m = 0$ and $db_m = z_m$, we have 
\begin{align}\label{zoom}
&\sum_j\| \nabla_{\frac{\p}{\p t^j}} b_m\|^2_{L^2}+  n^2\| b_m\|^2_{L^2}+\|d_N b_m\|^2_{L^2}+\| d_N^*b_m\|^2_{L^2}\nonumber\\
&
-c_n (\|\nabla b_m\|_{L^2}\sqrt{\sum_j\|e^*(dt^j)b_m\|^2_{L^2}}+\sum_j\|e^*(dt^j)b_m\|^2_{L^2})\leq \|z_m\|^2_{L^2}. 
\end{align}
Combining \eqref{zoom}  with   \eqref{coclprice4} yields 
\begin{align}\label{zoom2}
&\sum_j\| \nabla_{\frac{\p}{\p t^j}} b_m\|^2_{L^2}+  n^2\| b_m\|^2_{L^2}+\|d_N b_m\|^2_{L^2}+\| d_N^*b_m\|^2_{L^2}\nonumber\\
&\leq 2C_n (\|\nabla b_m\|_{L^2}\sqrt{\|z_m\|_{L^2}\|b_m\|_{L^2}}+\|z_m\|_{L^2}\|b_m\|_{L^2})+ \|z_m\|^2_{L^2}\nonumber\\
&\leq 2C_n (\alpha\|\nabla b_m\|_{L^2}^2+\frac{4+\alpha^2}{4}\|z_m\|_{L^2}\|b_m\|_{L^2})+ \|z_m\|^2_{L^2}. 
\end{align}
Choosing $\alpha$ sufficiently small, we can absorb the $2C_n\alpha\|\nabla b_m\|_{L^2}^2$ term into the lefthand side to conclude that there is some constant $\hat c_n>0$  such that 
$$\|b_m\|_{L^2}\leq \hat c_n\|z_m\|_{L^2}.$$
 The proof now proceeds as in the cases with $2<\frac{n_G(2)+2n_G(1)}{k-1 +\min\{k-1,n_G(2)\}}$.
\end{proof}

We remark that using the techniques of Section 7, we can use the vanishing of $\mathcal{H}_2^k(X)$, for $k\not = \frac{\text{dim}(X)}{2}$ as a starting point for an alternate  proof of (a slightly stronger version of) Theorem \ref{ThmG}.

\begin{remark}
In the papers \cite{BR20},\cite{BR23a}, and \cite{BR23b}, Bourdon and Remy prove numerous vanishing and nonvanishing theorems for $L_p$-cohomology of symmetric spaces. 
Their vanishing results have significant intersection with those obtained above and also include vanishing results for ranges of $k$ and $p$ that we have not considered in this note. As in Pansu's work, their vanishing theorems are to a significant degree founded on the construction of primitives obtained by integration along orbits of flows. Our estimates utilize infinitesimal data associated with these same flows. Hence one expects significant overlap in the results reached by the two methods.  
\end{remark}


\section{Acknowledgments}
I would like to thank the participants and the organizers of the AMS Special Session on the Singer-Hopf Conjecture in Geometry and Topology in the March 2023  Spring Southeastern Sectional Meeting for stimulating and encouraging this work.
I thank Luca Di Cerbo for numerous discussions on the contents of this note. I thank Marc Bourdon and Antonio Lopez for pointing out an error in an earlier draft of this paper. 
This work was partially supported by the Simons Foundation grant $\#$391000626.

\end{document}